\documentclass[11pt]{amsart}
\usepackage{amsmath}
\usepackage{bbm}
\usepackage{amssymb,graphicx,textcomp,bbm}
\usepackage{mathtools}
\usepackage{amsthm}

\usepackage[%pagebackref,
hypertexnames=false, colorlinks, citecolor=black, linkcolor=blue, urlcolor=red]{hyperref}

\newcommand{\cQ}{\mathcal{Q}}
\newcommand{\cR}{\mathcal{R}}
\newcommand{\cE}{\mathcal{E}}

\newcommand{\E}{\mathbb{E}}
\newcommand{\bE}{\mathbf{E}}

\newcommand{\R}{\mathbb{R}}
\newcommand{\cX}{\mathcal{X}}
\newcommand{\cD}{\mathcal{D}}
\newcommand{\cF}{\mathcal F}

\newcommand{\e}{\varepsilon}
\newcommand{\f}{\varphi}

\newcommand{\ch}{\operatorname{ch}}

\newcommand{\dd}{\mathrm d}

\newcommand{\1}{\mathbf{1}}

\newcommand{\ci}[1]{_{ {}_{\scriptstyle #1}}}

\newcommand{\ut}[1]{^{\scriptstyle \text{\rm #1}}}
\numberwithin{equation}{section}

\newtheorem{theorem}{Theorem}[section]
\newtheorem{lemma}[theorem]{Lemma}
\newtheorem{proposition}[theorem]{Proposition}

\theoremstyle{remark}
\newtheorem{remark}[theorem]{Remark}
\newtheorem*{remark*}{Remark}

\begin{document}
\title[Superexponential estimates and weighted lower bounds]{Superexponential estimates and weighted \\
lower bounds for the square function}
\author{Paata Ivanisvili and Sergei Treil}
\address{Department of Mathematics, Princeton University; MSRI; UC Irvine, USA}
\email{paatai@princeton.edu \textrm{(P.\ Ivanisvili)}}
\address{Department of Mathematics, Brown University, Providence, RI 02912, USA}
\email{treil@math.brown.edu \textrm{(S.\ Treil)}}
\thanks{This paper is  based upon work supported by the National Science Foundation under Grant No. DMS-1440140 while the authors were in residence at the Mathematical Sciences Research Institute in Berkeley, California, during the Spring  semester of 2017. The work of S.~Treil is also supported by the National Science Foundation under the grant  DMS-1600139}

\begin{abstract} 
We prove  the following superexponential distribution inequality: for  any integrable $g$  on $[0,1)^{d}$ with zero average, and any $\lambda>0$ 
\[
|\{ x \in [0,1)^{d} \; :\; g \geq\lambda \}| \leq e^{- \lambda^{2}/(2^{d}\|S(g)\|_{\infty}^{2})},
\]
where $S(g)$  denotes  the classical dyadic square function in $[0,1)^{d}$. The estimate is sharp when dimension $d$ tends to infinity in the sense that the constant $2^{d}$ in the denominator cannot be replaced by $C2^{d}$ with $0<C<1$ independent of $d$ when $d \to \infty$. 

For $d=1$ this is a classical result of Chang--Wilson--Wolff \cite{ChWiWo1985}; however, in the case $d>1$ they work with special square function $S_\infty$, and their result does not imply the estimates for the classical square function.  

Using good $\lambda$ inequalities technique we then obtain unweighted and weighted  $L^p$ lower bounds for $S$; to get the corresponding good $\lambda$ inequalities we need to modify the classical construction. 

We also show how to obtain our superexponential distribution inequality (although with worse constants) from the weighted $L^2$ lower bounds for $S$, obtained in \cite{DIPTV_2017}. 
%In particular, this extends Chang--Wilson--Wolff's superexponential bounds  to dyadic square functions in high dimensions.   
\end{abstract}
\maketitle
%\maketitle
\section{Introduction}
\subsection{Setup}
On a probability space $(\cX, \cF, \sigma)$ consider a discrete time atomic filtration, i.e.~a sequence of increasing sigma-algebras $\cF_n\subset \cF$, $n\ge 0$ such that  for each $\mathcal{F}_n$ there exists a countable collection $\mathcal{D}_n$ of disjoint sets with the property that every set of $\mathcal{F}_n$ is a union of sets in $\mathcal{D}_n$.  We assume that $\cF_0=\{\varnothing, \cX\}$ and that the $\sigma$-algebra $\cF$ is generated by $\cF_n$. 

A typical example one should have in mind is the standard dyadic filtration on the unit cube in $Q_0=[0,1)^d$ in $\R^d$; the collection $\cD_n$ in this case is the collection of the dyadic subcubes of $Q_0$ of size $2^{-n}$.

We denote $\cD=\bigcup_{n\ge 0}\cD_n$, and we will call the elements of $\cD$ \emph{atoms} or \emph{cubes} (by analogy with the dyadic filtration). We will often use the notation $|A|$ for $\sigma(A)$ and $\dd x$ for $\dd \sigma(x)$. 

For a cube $Q\in\cD_n$ we denote by $\ch Q$ the collection of its children, 
\[
\ch Q=\{R\in\cD_{n+1}: R\subset Q\}. 
\] 

For  $f \in L^{1}(\cX)$, and any measurable $A\subset \cX$ with $|A|>0$,  we define the average 
\begin{align*}
\langle f\rangle\ci{A} := \frac{1}{|A|}\int_{A} f; 
\end{align*}
if $|A|=0$ we just set $\langle f\rangle\ci{A}=0$.

We say that a filtration is \emph{$\alpha$-homogeneous}, $0<\alpha\le 1/2$,  if for any cube $Q$ and for any its child $Q'$
\[
|Q'| \ge \alpha |Q|. 
\]
The standard dyadic filtration in $\R^d$ is clearly $\alpha$-homogeneous with $\alpha=2^{-d}$. 

As a final remark, we can always assume without loss of generality that our probability space $\cX$ is the unit interval $[0,1)$, which gives another justification for the notation for the measure.  

\subsection{Square functions}
Define the expectation operators
\begin{align*}
\bE\ci Q f (x) := \langle  f \rangle\ci Q \1\ci Q(x), \qquad 
\bE_n f := \sum_{Q\in\cD_n} \bE\ci Q f , 
\end{align*}
and the martingale difference operators
\begin{align*}
\Delta\ci Q := \sum_{R\in\ch Q} \bE\ci R\quad -  \bE\ci Q, \qquad \Delta_n := \bE_{n+1} - \bE_n = \sum_{Q\in\cD_n} \Delta\ci Q. 
\end{align*}
%%
%For simplicity we will always assume that there exists a sufficiently large $N$ such that $f_{k}=f_{\bEll}$ for any $k,\ell \geq N$. 
Recall that the classical martingale  square function $S$ is defined as 
\begin{align}\label{dd1}
S f := \left(\sum_{n=0}^{\infty} |\Delta_n f|^{2} \right)^{1/2} = \Biggl(\sum_{Q\in\cD} |\Delta\ci Q f|^{2} \Biggr)^{1/2}.
\end{align}

One can define other non-classical square functions $S_p$, 
\begin{align*}
S_p f & := \Biggl( \sum_{Q\in\cD} \langle |\Delta\ci Q f|^p \rangle\ci Q^{2/p} \1\ci Q \Biggr)^{1/2}\\
S_\infty f & := \Biggl( \sum_{Q\in\cD} \|\Delta\ci Q f\|_\infty^2  \1\ci Q \Biggr)^{1/2}
\end{align*}
%%

%Notice that since after some $N$ we have $f_{N}=f_{N+1}=\ldots,$ the square function is well defined. Following \cite{ChWiWo1985} we define another square function, namely, 
%\begin{align*}
%\tilde{S}^{2}(f)(x) = \sum_{n=0}^{\infty} \|f_{n+1}-f_{n}\|^{2}_{L^{\infty}(Q_{n})} \mathbbm{1}_{Q_{n}},
%\end{align*}
%where $Q_{n}=Q_{n}(x)$ is the unique dyadic cube from $D(Q_{0})$ of length $2^{-n}$ which contains $x$. Clearly when the dimension $d=1$ the both square functions coincide. In general they are different, for example $S_{d}(f)$ can be uniformly  bounded whereas $\tilde{S}(f)$ will be unbounded (see the discussions below). 

In the case of dyadic filtration in $\R$ all the above square functions coincide; for more complicated filtration they are different. 
%%%
%\begin{align*}
%Sf(x) \le S_p f(x) \le S_q f(x) \le S_\infty f(x), \qquad 
%\end{align*}

In the case of $\alpha$-homogeneous filtration the square functions $S_p$ are equivalent in the sense of two sided pointwise estimate
\begin{align}
\label{S_p-equiv}
S_\infty f(x) \le C(\alpha, p) S_pf (x) \le S_\infty f(x). 
\end{align}
For general filtrations only trivial monotonicity relations hold, 
\[
S_p f(x) \le S_q f(x) \qquad p\le q. 
\]

\subsection{Superexponential estimates of Chang--Wilson--Wolff}
If $S_\infty f\in L^\infty$, then  the distribution function of $f$ decays superexponentially, i.e.~for all $\lambda\ge 0$
\begin{align}\label{distb}
\left|  \left \{ x  \in \cX\; : \; f(x) - \langle f \rangle\ci{\cX} \geq  \lambda  \right\} \right| \leq  e^{-\lambda^{2}/(2 \|S_\infty f\|_{\infty}^{2})} . 
\end{align}
This result was stated in  \cite[Theorem 3.1]{ChWiWo1985} for the classical dyadic filtration in $\R^d$ there, but the proof works for an arbitrary atomic  filtration.

The superexponential bound \eqref{distb} is the square function analog of the classical Gaussian concentration inequality. The proof of \eqref{distb} presented in \cite{ChWiWo1985} is a slight adaptation of the classical Hoeffding's inequality~\cite{hoef} -- well known in probability (see also \cite{azuma}).  Essentially the argument relies on the Hoeffding's lemma which says that if the random variable $X$ has mean zero and $X$ is bounded  almost surely then  
\begin{align}\label{hoeff}
\mathbb{E} e^{X}  \leq e^{\|X\|^{2}_{\infty}/2}.
\end{align}
In case of Chang--Wilson--Wolff's superexponential bound one uses Hoeffding's inequality \eqref{hoeff} with $X$ replaced by $t(f_{n+1}-f_{n})\mathbbm{1}_{Q}$ with $|Q|=2^{-dn}$.

The superexponential estimate \eqref{distb} allows Wilson~\cite{Wilson1989-2} (see the first lemma)  to obtain weighted $L^{p}$ estimates for the square function $S_\infty f$ in terms of the maximal function $f^*$, see \eqref{f^*} below for the definition.  Namely for any $0<p<\infty$ we have 
\begin{align}\label{wilson}
\int |f^*|^{p} w\dd x \underset{p}{\lesssim}  [w]^{p/2}_{\infty} \int (S_\infty  f)^{p} w \dd x. 
\end{align}
%In fact  \eqref{wilson} follows by the standard arguments of Wilson using good-$\lambda$ inequalities (see the proof of the first lemma in ~\cite{Wilson1989-2}).

In this paper we prove for the  $\alpha$-homogeneous filtration the inequalities \eqref{distb} and \eqref{wilson} for the classical square function $S$ instead of $S_\infty$.

For general atomic filtrations \eqref{wilson} fails, there is a counterexample for $p=2$ in \cite{DIPTV_2017}. 

%The positive answer would provide estimates of the type \eqref{wilson} by repeating absolutely the same good-$\lambda$ inequalities of Wilson in the same way. 

%The difference between $\tilde{S}$ and $S$ arises when the dimension $d$ is greater than $1$. Indeed, if we set 
%\begin{align*}
%\Delta_{Q}f :=(f_{n+1}-f_{n})\mathbbm{1}_{Q}
%\end{align*}
%for $Q \in D(Q_{0})$ with $|Q|=2^{-n}$, then both square functions can be rewritten as follows 
%\begin{align*}
%&S_{d}^{2}(f) = \sum_{Q \in D(Q_{0})} (\Delta_{Q} f)^{2}\mathbbm{1}_{Q},\\
%&\tilde{S}^{2}(f) = \sum_{Q \in D(Q_{0})} \|\Delta_{Q} f\|_{L^{\infty}(Q)}^{2} \mathbbm{1}_{Q}.
%\end{align*}

\subsection{An example}\label{primery}
Let us present a function $f$ on a unit cube $Q_0=[0, 1)\times [0, 1)$ such that $Sf \in L^\infty$ but $S_\infty f$ is unbounded. 

For an interval $I$ let $I_-$ and $I_+$ be its left and right halves respectively. Let $h\ci I$ denotes  the $L^\infty$ normalized Haar function, $h\ci I = \1\ci{I_+}-\1\ci{I_-}$.  For a dyadic square $Q=I^1\times I^2$ define 
\begin{align*}
f\ci Q (x) := 1_{I^1_+}(x_1) h_{I^2} (x_2), \qquad x=(x_1,x_2)\in \R^2. 
\end{align*}
Let $\cQ_k$, $k=1, 2, \ldots$ be the collection of all squares of the form
\begin{align*}
Q=[0, 2^{-k})\times I, 
\end{align*}
where $I$ runs over all dyadic subintervals of $[0,1)$ of length $2^{-k}$, and let $\cQ:= \bigcup_{k\ge 0} \cQ_k$.  

Define $f:= \sum_{Q\in\cQ} f\ci Q$. Clearly 
\begin{align*}
\Delta\ci Q f = \left\{ \begin{array}{ll}
f\ci Q \qquad & Q\in \cQ , \\ 0 &Q\notin \cQ .
\end{array}   \right. 
\end{align*}
It is easy to see that $Sf=\1\ci{Q_0}$, but $S_\infty f$ is unbounded (goes to $\infty$ as $x_1\to 0$) because the squares $Q\in\cQ$ have unbounded overlap.

%Consider the square $[0,1]^{2}$ (here $d=2$). Consider its dyadic subdivision $D([0,1]^{2})$. Any $Q \in  D([0,1]^{2})$ has four children, and lets enumerate them in the order $Q_{I}, Q_{II}, Q_{III}, Q_{IV}$ in corresponding  to the partition of the standard Cartesian coordinate system.
%
%For the construction of such $f$ we need the following simple  rules. For $Q \in D([0,1]^{2})$ we say that $\Delta_{Q} f$ is {\em nontrivial} if
%\begin{align*}
%\Delta_{Q} f |_{Q_{I}}=1; \quad \Delta_{Q} f|_{Q_{II}}=\Delta_{Q} f|_{Q_{III}}=0; \quad \Delta_{Q} f|_{Q_{IV}}=-1. 
%\end{align*}
%Whenever $\Delta_{Q} f$ is nontrivial we will set $\Delta_{\tilde{Q}} f =0$ for any dyadic $\tilde{Q} \subset Q_{I} \cup Q_{IV}$, i.e., $f$ will be automatically defined on the right sided children of $Q$. On the other hand for nontrivial $\Delta_{Q} f$ we define  $\Delta_{Q_{II}} f$ and $\Delta_{Q_{III}} f$ to be nontrivial as well. Finally we start with $f$ such that $\Delta_{[0,1]^{2}} f$ is nontrivial, and  following our rules this defines  the collection $\{ \Delta_{Q} f\}_{Q \in D([0,1]^{2})}$ uniquely. Clearly $Sf =1$ whereas  $\tilde{S} f$ is unbounded close to the vertical segment $\{0\} \times [0,1]$.

The example illustrates that in higher dimensions   one may not expect the bound \eqref{wilson} but we will show that this is not the case. 
%\begin{theorem}\label{kalash}
%For all $\lambda\geq 0$ we have  
%\begin{align}\label{distb1}
%\frac{1}{|Q_{0}|}\left|  \left \{ x  \in Q_{0}\; : \; f(x) - \langle f \rangle_{Q_{0}} \geq  \lambda  \right\} \right| \leq  e^{-\lambda^{2}/(2^{d} \|S_{d} f\|_{\infty}^{2})}.
%\end{align}
%\end{theorem}
%
%
%
%In fact we will obtain bounds (\ref{distb1}) for a big class of square functions corresponding to {\em $p$-homogeneous} filtration  (see the next section for the definitions).

%We could not find the reference for the estimate (\ref{bilyk}) for  dyadic square function on $\mathbb{R}^{n}$. Therefore we decided to write our own proof of Theorem~\ref{kalash}.

\section{Main results}

For typographical reasons we will sometimes  use the symbol $\E f$ for the average over the whole space $\cX$, $\E f := \langle f\rangle\ci\cX$. 

The main result of the paper is the following theorem, generalizing the Chang--Wilson--Wolff estimate \eqref{distb} to  the standard martingale square function \eqref{dd1} in the homogeneous filtration. 

\begin{theorem}\label{trpi0}
For an $\alpha$-homogeneous filtration, any $f\in L^1$ with $S(f) \in L^{\infty}$ we have   
\begin{align}\label{trpi1}
\mathbb{E}   \exp\left(f- {S(f)^{2}}/{4\alpha} \right) \leq \exp\left(\mathbb{E} f\right).
\end{align}
\end{theorem}
%We will show that $A(n)$ can be any number (desirably largest)  such that 
%And if we choose the largest then  $A(n) \sim \frac{1}{n}$ as $n \to \infty$. 
The theorem immediately gives the estimate 
\begin{theorem}\label{Slavin}
Assume that for an $\alpha$-homogeneous filtration and $f\in L^1$ we have $\|S(f)\|_{\infty} < \infty$. Then 
\begin{align*}
\mathbb{E} \exp(f - \E f) \leq e^{\|S(f)\|^{2}_{\infty}/4\alpha}.
\end{align*}
\end{theorem}
It follows from Markov's inequality that 
\begin{align*}
|\{ x \in \cX: f(x) -\E f > \lambda \}| =|\{x\in \cX: \exp(t(f(x) -\E f)) > \exp(t\lambda)\}| \leq  e^{t^{2}\|S(f)\|^{2}_{\infty}/4\alpha - t\lambda }
\end{align*}
Optimizing over all $t$ (the minimum is attained at $t= 2\alpha\lambda/\|Sf\|_\infty^2$) we obtain the following superexponential bound for the distribution function.   
\begin{theorem}
	\label{t:genCWW}
Under assumptions of Theorem  \ref{Slavin} we have  for any $\lambda \geq 0$ %, and any $k\geq 1$  we have 
\begin{align}\label{ext}
\left| \{ x\in \cX: f(x) -\E f > \lambda    \}  \right| \leq  e^{-\alpha\lambda ^{2}/ \|S f\|_{\infty}^2}.
\end{align}
\end{theorem}

Recall that the martingale maximal function $f^*$ is defined as 
\begin{align}
\label{f^*} 
f^*(x) := \sup \{ |\bE_n f (x)| : {n\ge 0} \}. 
\end{align}

As a corollary of Theorem \ref{t:genCWW} we get the following distribution inequality for $f^*$, see Section \ref{s:ProofOf_genCWWmax} for the proof. 

\begin{theorem}
	\label{t:genCWWmax}
	Let for an $\alpha$-homogeneous filtration and a real-valued function $f\in L^1$ we have $Sf\in L^\infty$ and $\E f =0$. Then 
	\begin{align}
	\label{distr_of_psi^*}
	\left| \{ x\in \cX: f^*(x)  > \lambda    \}  \right| \leq 2 e^{-\alpha\lambda ^{2}/ \|S f\|_{\infty}^2}.
	\end{align}
\end{theorem}

Using good $\lambda$ inequalities we get from the above Theorem \ref{t:genCWWmax} the following lower bound for the martingale square function $S$. 

Recall that a weight (i.e.~a non-negative integrable function) $w$ is said to satisfy the \emph{martingale} $A_\infty$-condition if
\begin{align}
\sup \left\{ {\langle (1\ci Q w)^* \rangle\ci Q}/{\langle w \rangle\ci Q} \right\} =: [w]\ci{A_\infty}<\infty;
\end{align}
the number $[w]\ci{A_\infty}$ is called the $A_\infty$ characteristic of the weight $w$. 

In this paper we will skip the word \emph{martingale}, and just say ``$A_\infty$ weight''.

%In particular for the martingale square function (\ref{dd1}) on $\mathbb{R}^{d}$ we obtain Theorem~\ref{kalash}. This,  together with good-$\lambda$ inequalities obtained in the same way as in \cite{Wilson1989-2}, gives the weighted lower bounds for the dyadic square function $S_{d}$  in terms of a maximal function. 

\begin{theorem}
	\label{t:weighted}
	For an $\alpha$-homogeneous filtration let  $w$ be an $A_\infty$ weight. Then for any $f\in L^1$ and any $p$, $0< p<\infty$
	\begin{align}
	\label{weighted_f-S}
	\|f^*\|\ci{L^p(w)} \le C(\alpha, p)[w]\ci{A_\infty}^{1/2}\|Sf\|\ci{L^p(w)}, \qquad C(\alpha, p)= 2^{1/p}2^{(p+4)/2}\alpha^{-3/2} (\ln(2/\alpha))^{1/2}. 
	\end{align}
\end{theorem}

%\begin{theorem}
%For any $p$, $0<p<\infty$, any $w \in A_{\infty}(Q_{0})$, and any $f \in L^{1}(Q_{0})$ with $Mf \in L^{p}(w, Q_{0})$ we have 
%\begin{align*}
%\| Mf \|_{L^{p}(w, Q_{0})} \underset{d,p}{\lesssim}  [w]^{1/2}_{\infty} \|S_{d} f\|_{L^{p}(w,Q_{0})},
%\end{align*}
%where $Mf$ denotes the dyadic Hardy--Littlewood maximal operator on $L^{1}(Q_{0})$.   
%\end{theorem}

\section{Proof of superexponential estimates}
\subsection{Proof of Theorem~\ref{trpi0}}
Notice that it is enough to  prove Theorem~\ref{trpi0} for $f_{N}:=\mathbb{E}_{N}f$ instead of $f$ where $N$ is an arbitrary nonnegative integer. Indeed, since $\mathbb{E}f_{N} = \mathbb{E}f$ and $-S(f)^{2} \leq -S(f_{N})^{2}$, estimate  \eqref{trpi1} for $f_{N}$ implies that  
\begin{align*}
\mathbb{E} \exp(f_{N} - S(f)^{2}/4\alpha) \leq \exp(\mathbb{E}f)
\end{align*}
It follows that the sequence $g_{N}:=\exp(f_{N}-S(f)^{2}/4\alpha)$ is uniformly integrable, for example $g_{N} \in L^{2}$. Using uniform integrability of $g_{N}$ and convergence in measure as $N\to \infty$ we obtain 
$$
\mathbb{E} \exp(f - S(f)^{2}/4\alpha) \leq \exp(\mathbb{E}f).
$$

In what follows we assume that $f=f_{N}$. In this case we have 
\begin{align*}
&f_{N} = f_{0} + \sum_{k=0}^{N-1} (f_{k+1}-f_{k});\\
&S^{2}(f_{N})  = \sum_{k=0}^{N-1} (f_{k+1}-f_{k})^{2}. 
\end{align*}
We will set $S(f_{0})=0$.

Define 
\begin{align*}
U_{\alpha}(x,y)=e^{x-\frac{y^{2}}{4\alpha}} \quad \text{for all} \quad (x,y) \in \mathbb{R}\times\mathbb{R}_{+}. 
\end{align*}
It is enough to show that  
\begin{align}\label{iteracia}
\mathbb{E} U_{\alpha}(f_{N},S(f_{N})) \leq \mathbb{E} U_{\alpha}(f_{N-1},S(f_{N-1}))
\end{align}
Indeed, iterating the inequality \eqref{iteracia} we will obtain 
\begin{align*}
\mathbb{E} U_{\alpha}(f_{N},S(f_{N})) \leq \ldots \leq \mathbb{E} U_{\alpha}(f_{0}, S(f_{0})) = \mathbb{E} U_{\alpha}(\mathbb{E}f, 0) = \exp(\mathbb{E}f)
\end{align*}
which proves the theorem. 

To prove \eqref{iteracia} we use the identity $\mathbb{E}\,  \bE_{N-1}= \mathbb{E}$, so we only need to show that
$$
\bE_{N-1}U_{\alpha}(f_{N}, S(f_{N})) \leq U_{\alpha}(f_{N-1},S(f_{N-1})).
$$
The latter estimate simplifies as follows:  for each $Q \in \mathcal{D}_{N-1}$ we have 
\begin{align*}
\frac{1}{|Q|}\int_{Q} U_{\alpha}(f_{N}, S(f_{N})) \leq U_{\alpha}(f_{N-1}(x),S(f_{N-1})(x)) \quad \text{for all} \quad  x \in Q. 
\end{align*}
Since 
\begin{align*}
U_{\alpha}(f_{N}, S(f_{N})) =  U_{\alpha}\left(f_{N-1}+(f_{N}-f_{N-1}), \sqrt{S^{2}(f_{N-1}) + (f_{N}-f_{N-1})^{2}}\right)
\end{align*}
we see that $f_{N-1}$ and $S(f_{N-1})$ are   constant  on $Q$. The difference $f_{N}-f_{N-1}$ takes values $c_{j}:=\langle f\rangle_{Q_{j}} - \langle f \rangle_{Q}$ on each  $Q_{j} \in \mathrm{ch}\, Q$,  $j=1,\ldots, \#(\mathrm{ch}\, Q)$, where $\#(\mathrm{ch}\, Q)$ denotes number of children in $Q$. If we set $p_{j}:=\frac{|Q_{j}|}{|Q|}\geq \alpha$ then it follows that $\sum_{j} p_{j} c_{j}=0$.

Thus to prove the theorem we only need to show that for any family of points $\{ c_{j}\}  \subset \mathbb{R}$,   positive numbers $\{p_{j}\} \subset \mathbb{R}_{+}$ with $p_{j} \geq \alpha$,  $\sum_{j} p_{j}=1$, and $\sum_{j} p_{j} c_{j}=0$  we have 
\begin{align}\label{heat}
\sum_{j} p_{j} U_{\alpha}(x+c_{j}, \sqrt{y^{2}+c_{j}^{2}})  \leq U_{\alpha}(x,y)
\end{align}
for all $(x,y) \in \mathbb{R}\times \mathbb{R}_{+}$.

%Indeed, let $S_{m+1}(f)^{2} = d_{1}^{2}+\ldots+d_{m+1}^{2}$ and $f_{m+1}=d_{1}+\ldots+d_{m+1}$. Consider the atom $D_{m}$ where $(d_{1}, \ldots, d_{m})$ is constant, and $d_{m}$ takes values $c_{j}$ with probability $p_{j}$  on this atom with  $\int_{D_{m}} d_{m+1}=0$.   
%\begin{align*}
%&\mathbb{E} \left( U_{p}(f_{m+1}, S_{m+1}(f)) | D_{m}\right)=\sum_{j}  p_{j}U_{p}(f_{m}+c_{j}, \sqrt{S_{m}^{2}(f)+c_{j}^{2}}) \\
%&\leq U_{p}(f_{m}, S_{m}(f)) \quad \text{on} \quad D_{m}. 
%\end{align*}
%Iterating the process we obtain $\mathbb{E}\,  U_{p}(f, S(f)) \leq U_{p}(\mathbb{E} f, 0)$ which proves (\ref{trpi1}). 

Inequality \eqref{heat} %simplifies to the following one
follows from the lemma below. 
\begin{lemma}
For all $p_{j}\geq \alpha \in (0,1/2]$ with $\sum_{j} p_{j}=1$, and any $c_{j} \in \mathbb{R}$ with $\sum_{j} p_{j} c_{j}=0$ we have 
\begin{align*}
\sum p_{j} e^{c_{j} - \frac{c_{j}^{2}}{4\alpha}}\leq 1.
\end{align*}
\end{lemma}
\begin{proof}
Consider $f(c):=e^{c-\frac{c^{2}}{4\alpha}}$. Notice that 
\begin{align*}
f''(c) = \frac{f(c)}{4\alpha^{2}}\left((c-2\alpha)^{2}-2\alpha\right).
\end{align*}
Therefore, $f(c)$ is concave on $I_{\alpha}:=[2\alpha-\sqrt{2\alpha}, 2\alpha+\sqrt{2\alpha}]$, and convex on $I_{\alpha}^{\pm}$ where $I_{\alpha}^{+} = (2\alpha+\sqrt{2\alpha}, \infty)$ and $I_{\alpha}^{-}=(-\infty, 2\alpha-\sqrt{2\alpha})$. Consider our family of points $\{c_{j}\}\subset \mathbb{R}$. Without loss of generality we can assume that if there are points from $\{c_{j}\}$ which belong to $I_{\alpha}$  then they are equal to each other. Indeed, suppose there are points $c_{1}, c_{2} \in I_{\alpha}$ such that $c_{1}\neq c_{2}$. Consider new pairs $c_{1}^{*}=c_{2}^{*}=\frac{p_{1}c_{1}+p_{2}c_{2}}{p_{1}+p_{2}}$. Clearly $c^{*}_{1}, c_{2}^{*}\in I_{\alpha}$, and by concavity we have 
\begin{align*}
p_{1}f(c_{1})+p_{2}f(c_{2})\leq p_{1}f(c_{1}^{*})+p_{2}f(c_{2}^{*}). 
\end{align*} 
Since $p_{1}c_{1}+p_{2}c_{2}=p_{1}c_{1}^{*}+p_{2}c_{2}^{*}$ we see that by replacing the points $c_{1}, c_{1}$ with $c_{1}^{*}, c_{2}^{*}$ we only increase the value of $p_{1}f(c_{1})+p_{2}f(c_{2})$. Thus all points from $I_{\alpha}$ should coincide. In other words we can assume that we have only one point in $I_{\alpha}$ with big weight $p_{m_{1}}+\ldots+p_{m_{2}}$. 

Next, consider those points $c_{j}$ which belong to $I_{\alpha}^{+}$. We claim that if there are such points then we can assume that  there is only one such point.   Indeed, let $c_{1}, c_{2} \in I_{\alpha}^{+}$ be such that $c_{1}<c_{2}$. Pick any number $t>0$ such that $c_{1}-t \in I_{\alpha}^{+}$. Consider $c_{1}^{*} = c_{1}-t$ and $c_{2}^{*} = c_{2}+tp_{1}/p_{2}$. Clearly $c_{1}^{*}<c_{1}<c_{2}<c_{2}^{*}$, and $p_{1}c_{1}+p_{2}c_{2}=p_{1}c_{1}^{*}+p_{2}c_{2}^{*}$. It follows from convexity that 
\begin{align*}
p_{1}f(c_{1})+p_{2}f(c_{2})\leq p_{1}f(c_{1}^{*})+p_{2}f(c_{2}^{*}).
\end{align*}
Thus moving points $c_{1}, c_{2}$ apart from each other we only increase the value $p_{1}f(c_{1})+p_{2}f(c_{2})$. Therefore at some moment $c_{1}=2\alpha+\sqrt{2\alpha} \in I_{\alpha}$.

In a similar way we can assume that if there are points in $I_{\alpha}^{-}$ then there is only one such point.

Next, suppose we have a point $c_{1}$ which belongs to $I_{\alpha}^{-}$ and $c_{m} \in I_{\alpha}^{+}$. We claim that by moving $c_{1}$ and $c_{m}$ close to each other and keeping the quantity $p_{1}c_{1}+p_{m}c_{m}$ the same we will only increase the value $p_{1}f(c_{1})+p_{m}f(c_{m})$. Indeed, the claim follows from the observation that the function $f(c)$ is symmetric with respect to the point $c=2\alpha$, and $f(c)$ is decreasing for $c\geq 2\alpha$.  Thus, by moving points $c_{1}, c_{m}$ close to each other we will reach the position when one of the points $c_{1}, c_{m}$ belong to $I_{\alpha}$. 

Finally to prove the lemma we need to consider only 3 cases: 1) when we have only one point  $c_{1}$ which belongs to  $I_{\alpha}$; 2) When we have two points  $c_{1}, c_{2}$  with  $c_{1}\in I_{\alpha}$, and $c_{2}\in I_{\alpha}^{+}$;  3) When we have $c_{1}, c_{2}$ with $c_{1}\in I_{\alpha}$ and $c_{2} \in I_{\alpha}^{-}$.  

In the first case the condition $\sum_{j} p_{j}c_{j}=0$ implies that $c_{1}=0$, and therefore the lemma is trivial. 

In the second case since $p_{1}c_{1}+p_{2}c_{2}=0$ we can assume that $c_{1}<0$ and $c_{2} >2\alpha+\sqrt{2\alpha}$. Moving points $c_{1}, c_{2}$ close to each other and keeping the equality $p_{1}c_{1}+p_{2}c_{2}=0$ we will only increase the value $p_{1}f(c_{1})+p_{2}f(c_{2})$.  Using this procedure we can reach the position when $c_{2}\in I_{\alpha}$ which reduces to the first case. 

In the third case we  need to prove that if $p_{1}+p_{2}=1$, $p_{1}, p_{2}\geq 0$,  $c_{1}p_{1}+c_{2}p_{2}=0$, $p_{1}, p_{2}\geq \alpha$, $\alpha\leq 1/2$, $c_{2}\leq 2\alpha-\sqrt{2\alpha}$, $c_{1}\in I_{\alpha}$ we have 
\begin{align*}
p_{1}e^{c_{1}-\frac{c_{1}^{2}}{4\alpha}}+p_{2}e^{c_{2}-\frac{c_{2}^{2}}{4\alpha}}\leq 1. 
\end{align*}
We can also assume that $p_{2}=\alpha$. Indeed, if $p_{2}>\alpha$ then we can decompose $p_{2}=\alpha+p^{+}$, and think that we have two points coinciding at $c_{2}$ with weights $\alpha$ and $p^{+}$. Then repeating the previous discussion, i.e., moving them apart from each other, we will arrive to the conclusion that $p_{2}=\alpha$.

 Without loss of generality we can assume that $2\alpha \geq c_{1}\geq 0$. Indeed, $c_{1}<0$ would contradict to the assumption $p_{1}c_{1}+p_{2}c_{2}=0$. If $c_{1}$ were greater than $2\alpha$ then we would move points $c_{2}, c_{1}$ close to each other keeping the condition $p_{1}c_{1}+p_{2}c_{2}=0$ and increasing the value $p_{1}f(c_{1})+p_{2}f(c_{2})$ (we remind that $f$ is decreasing on $(2\alpha, \infty)$ and increasing on $(-\infty, 2\alpha)$).  
By making change of variables $a_{j}:=1-\frac{c_{j}}{2\alpha}$ we need to show that 
\begin{align}\label{mokvda}
(1-\alpha)e^{(1-a_{1}^{2})\alpha}+\alpha e^{(1-a_{2}^{2})\alpha}\leq 1
\end{align}
provided that $(1-\alpha)a_{1}+\alpha a_{2}=1$, and $0\leq a_{1}\leq1,$  $\frac{1}{\sqrt{2\alpha}}\leq a_{2}$ (the assumption $a_{1}\leq 1$ comes from the fact that  $(1-\alpha)a_{1}+\alpha a_{2}=1$). In this case \eqref{mokvda} follows from the following Proposition~\ref{dzili} where we set $n:=1/\alpha \geq 2$, $a:=a_{1}$ and $a_{2} =  n-(n-1)a$. 

Thus, the lemma, and so Theorem \ref{trpi0} are proved modulo Proposition \ref{dzili} below. 
\end{proof}

 \begin{proposition}\label{dzili}
 For any $n\geq 2$ we have
  \begin{align*}
 \max_{a \in [0,1]} \; (n-1)e^{\frac{1-a^{2}}{n}} + e^{\frac{1-(n(1-a)+a)^{2}}{n}} \leq n.
 \end{align*}
  \end{proposition}
 \begin{proof}

 Consider 
 \begin{align*}
 f(a) := (n-1)e^{\frac{1-a^{2}}{n}} + e^{\frac{1-(n(1-a)+a)^{2}}{n}} -n \quad \text{for} \quad a \in [0,1].
 \end{align*}
 Notice that 
 \begin{align*}
 f'(a) = 2(n-1) \left(e^{\frac{1-(n(1-a)+a)^{2}}{n}}(n(1-a)+a) -a e^{\frac{1-a^{2}}{n}} \right).
 \end{align*}
 Therefore $f(1)=f'(1)=0$ and $f''(1)=-\frac{2(n-1)(n-2)}{n} <0$. Also notice that 
 \begin{align}\label{c5}
 f(0) = (n-1)e^{1/n}+e^{(1-n^{2})/n}-n <0.
 \end{align}
Indeed, we rewrite  \eqref{c5} as $ 1-\frac{1}{n}+\frac{e^{-n}}{n}<e^{-1/n}$, and we use estimate $e^{-1/n}\geq 1-1/n+1/2n^{2}-1/6n^{3}$. Then our claimed inequality would follow from $(6e^{-n}n^{2}-3n+1)/(6n^{3})<0$ which is true because $6e^{-n}n^{2}\leq 24e^{-2}<4\leq 3n-1$ for $n\geq 2$ (notice that $e^{2} > 6$).

Thus  the  only interesting case is what happens with $f(a_{0})$ at critical points $a_{0}$, i.e., $f'(a_{0})=0$. Consider $f'(a_{0})=0$. The equation can be simplified as follows 
 \begin{align*}
 e^{\frac{1-(n(1-a_{0})+a_{0})^{2}}{n}}=\frac{a_{0}}{n(1-a_{0})+a_{0}} e^{\frac{1-a_{0}^{2}}{n}}.
 \end{align*}
 Substituting into the expression for $f(a)$ we see that it would be sufficient to show 
 \begin{align*}
 (n-1)e^{\frac{1-a_{0}^{2}}{n}} + \frac{a_{0}}{n(1-a_{0})+a_{0}} e^{\frac{1-a_{0}^{2}}{n}}  -n \leq 0.
 \end{align*}
 Simplifying further it is enough to show the following
 \begin{align*}
 1 - \frac{1-a_{0}}{a_{0}+n(1-a_{0})} \leq e^{\frac{a_{0}^{2}-1}{n}}.
 \end{align*}
 We will show that this inequality holds in fact for all $a_{0} \in [0,1]$. %We consider two cases separately: 1) when $n\in [2,3]$;  2)  when $n\geq 3$. In the first case 
 We estimate the right hand side by $e^{-x}\geq 1-x+x^{2}/2-x^{3}/6$. Therefore it would suffice to show that 
 \begin{align*}
 0\leq \frac{1-a_{0}}{a_{0}+n(1-a_{0})} + \frac{a_{0}^{2}-1}{n} + \frac{(a_{0}^{2}-1)^{2}}{2n^{2}}+\frac{(a^{2}_{0}-1)^{3}}{6n^{3}}.
 \end{align*}
 If we denote $w=1-a^{2} \in [0,1]$ the latter inequality simplifies as follows 
 \begin{align*}
 -\frac{w^{2}}{6n^{3}} &+ \frac{w}{2n^{2}}  - \frac{1}{n} + \frac{1}{nw+1-w+\sqrt{1-w}} \\
 &\ge-\frac{w^{2}}{6n^{3}} + \frac{w}{2n^{2}}  - \frac{1}{n} + \frac{1}{nw+1-w+1-w/2}  \\
 &=\frac{12(1-w)n^{3}+6(w-1)(w+4)n^{2}+(-2w^{3}-9w^{2}+12w)n+3w^{3}-4w^{2}}{6n^{3}(w(2n-3)+4)}
 \end{align*}
 The second derivative in $n$ of the numerator of the last expression is $12(1-w)(6n-4-w)$ which is positive for $n\geq 2$. So the numerator is convex in $n$, and we will estimate from below by its tangent line at point $n=2$. The tangent line has the expression 
 \begin{align*}
 w^{2}(2-w)+(-2w^{3}+15w^{2}-60w+48) (n-2)
 \end{align*}
 The first term is positive since $w \in [0,1]$. The coefficient in front of $(n-2)$ is decreasing in $w$, and attains its minimal value when $w=1$ which is positive.  
% To verify the second case, i.e., $n\geq 3$, by  denoting $\frac{a_{0}}{1-a_{0}}=t\geq 0$ it suffices to show 
 %\begin{align}\label{axali}
 %\ln \left(1-\frac{1}{t+n} \right) +\frac{2t+1}{n(t+1)^{2}}\leq 0
 %\end{align} 
% We can use estimate  $\ln(1-x) \leq -x$ then the left hand side of the inequality would become $-\frac{t(t(n-2)-1)}{(t+n)n(t+1)^{2}}$ which has negative sign provided that $t \geq \frac{1}{n-2}$. Since $n\geq 3$ without loss of generality we can assume that $t \leq 1$ (otherwise $1 \geq \frac{1}{n-2}$). In this case we can rewrite (\ref{axali}) as follows 
 %\begin{align}\label{c1}
 %\left( 1-\frac{1}{t+n}\right)^{n} \leq \exp\left(-\frac{2t+1}{(t+1)^{2}} \right)
 %\end{align}
 %Given $x \geq  -1$ the map $(1+\frac{x}{u})^{u}$ is increasing in $u \geq 1$ and it tends to $e^{x}$. Therefore we can estimate 
 %\begin{align*}
 %\exp\left( -\frac{2t+1}{(t+1)^{2}}\right)\geq \left( 1-\frac{2t+1}{(t+1)^{2}2n}\right)^{2n}
 %\end{align*}
 %Therefore (\ref{c1}) would follow from $\left( 1-\frac{2t+1}{(t+1)^{2}2n}\right)^{2} \geq \left( 1-\frac{1}{t+n}\right)$. The latter simplifies to 
 %\begin{align}\label{numer}
 %\frac{(4t^{4}+8t^{3}+4t^{2})n^{2}+(-8t^{4}-20t^{3}-12t^{2}+1)n+4t^{3}+4t^{2}+t}{4(t+1)^{4} n^{2} (t+n)}\geq 0
 %\end{align}
% Numerator as a function of $n\geq 3$ is increasing in $n$. Indeed, its derivative  in $n$ is estimated from below (using $n\geq 3$ as follows)
% \begin{align*}
% 6(4t^{4}+8t^{3}+4t^{2})+(-8t^{4}-20t^{3}-12t^{2}+1)=16t^{4}+28t^{3}+12t^{2}+1 >0. 
 %\end{align*}
% On the other hand numerator (\ref{numer}) at point $n=3$ equals $12t^{4}+16t^{3}+4y^{2}+t+3 >0$. 
 \end{proof}
 
% This finishes the proof of the lemma and the theorem. 
% \end{proof}
 %\begin{proposition}\label{ura}
% We have 
% \begin{align*}
% \frac{n-1}{n}+\frac{e^{-n}}{n}<e^{-1/n}
% \end{align*}
 %for $n\geq 3$. 
 
% \end{proposition}
% \begin{proof}
%We estimate $e^{-1/n}\geq 1-1/n+1/2n^{2}-1/6n^{3}$. Then our inequality simplifies to $(6e^{-n}n^{2}-3n+1)/(6n^{3})<0$ which is true because $6e^{-n}n^{2}\leq 8\leq 3n-1$ for $n\geq 3$ (notice that $e^{3} > 20$). 
% \end{proof}

 %We have 
 %\begin{align*}
 %F(0)=2e^{A(1-t^{2})}, \quad F(t) = e^{A}+e^{A(1-4t^{2})}.
 %\end{align*}
 %Let 
 %\begin{align*}
 %C(t):=F(0) - F(t).
 %\end{align*}
 %Clearly 
 %\begin{align*}
 %C'(t) >0 \quad \text{on} \quad \left(0,\frac{\ln 2}{ \sqrt{3A}} \right),
 %\end{align*}
 %and it is negative on  $(\frac{\ln 2}{ \sqrt{3A}}, \infty)$. Since $C(0)=0$ we see that there exists a unique point $t_{0}>0$ where $C(t_{0})=0$ and so that $C(t)>0$ on $(0,t_{0})$ and it is negative on $(t_{0}, \infty)$. Also notice that $t_{0} = \sqrt{-\frac{\ln x_{0}}{A}}$ where $x_{0}\approx 0.543...$ is the unique positive solution of the equation $x^{3}+x^{2}+x-1=0$.  To summarize we have
 %\begin{align*}
% &F(0)> F(t) \quad \text{for} \quad t\in \left( 0,\sqrt{-\frac{\ln x_{0}}{A}}\right);\\
% &F(0)<F(t) \quad \text{for}\quad t\in \left(\sqrt{-\frac{\ln x_{0}}{A}}, \infty\right).
 %\end{align*}

\subsection{Proof of Theorem \ref{t:genCWWmax}}
\label{s:ProofOf_genCWWmax}
Proof of this theorem is well-known and simple: we present it here only for the readers' convenience. 

Given $\lambda>0$ let $\cR=\cR_\lambda \subset\cD$ be the collection of maximal cubes such that 
\begin{align*}
|\langle f\rangle\ci Q | >\lambda . 
\end{align*}
Define function $\tilde f $ by replacing $f$ on  cubes $R\in\cR$ by $\langle f\rangle\ci R$, i.e.~as 
\begin{align*}
\tilde f = f -\sum_{R\in\cR} (f - \langle f \rangle\ci R )\1\ci R.
\end{align*}
It is easy to see that 
\begin{align*}
\{x\in\cX :     f^*(x)>\lambda  \} = \{x\in\cX :     |\tilde f(x) | >\lambda  \} =\bigcup_{R\in\cR} R,  
\end{align*}
and that $S\tilde f(x) \le Sf(x)$ a.e. Applying Theorem \ref{t:genCWW} to $f$ and $-f$ we get \eqref{distr_of_psi^*}. \hfill\qed

\section{Weighted estimates}

\subsection{Unweighted and weighted good \texorpdfstring{$\lambda$}{lambda} inequalities}
The following lemma was proved in \cite[Corollary 3.1]{ChWiWo1985} (with $1/2$ instead of $\alpha^3/(1-2\alpha+2\alpha^2)$) for the square function $S_\infty$. The proof for the classical square function $S$ requires some modifications, so for the reader's convenience we present it here. 
\begin{lemma}
	\label{l:Good_la-01}
	For an $\alpha$-homogeneous filtration and $f\in L^1$, and $\lambda>0$,  let $\cQ=\cQ(\lambda)$ be the collection of maximal cubes $Q\in\cD$ such that $|\langle f \rangle\ci Q |>\lambda$. 
	
	Then for each $Q\in\cQ$  we have the following good $\lambda$ inequality
\begin{align}
\label{goog_la-01}
|\{x\in Q:  f^*(x)  >2\lambda,  Sf(x) & \le \e\lambda \}| 
\\  \notag
& \le 2 \exp\left(- \frac{\alpha^3 (1-\e)^2}{ (1-2\alpha+2\alpha^2)\e^2}    \right) |Q|
%\left| \{ x\in Q: f^*(x)>\lambda \}      \right|
\end{align}
\end{lemma}
\begin{proof}
Consider the family $\cQ$ of maximal cubes	$Q\in\cD$ such that 
\begin{align*}
|\langle f \rangle\ci Q |>\lambda. 
\end{align*}
For each $Q\in\cQ$ consider the set
\begin{align*}
E\ci Q:= \{x \in Q : Sf(x)\le \e\lambda , f^*(x)>2\lambda \}. 
\end{align*}
Note that if $E\ci Q\ne \varnothing$, then $|\langle f \rangle\ci Q | \le (1+\e)\lambda$: indeed, if $\widehat Q$ is the parent of $Q$, then by the construction $|\langle f\rangle\ci{\widehat Q}|\le \lambda$ and for $x\in E\ci Q$
\[
|\langle f\rangle\ci{\widehat Q} - \langle f\rangle\ci{ Q} | \le Sf(x) \le\e\lambda.
\]

The inequality $|\langle f \rangle\ci Q | \le (1+\e)\lambda$ implies that if we define $f_1:= (f - \langle f \rangle\ci Q ) \1\ci{Q}$, then
\begin{align*}
E\ci Q \subset E^1\ci Q := \{x\in Q:  f_1^*(x)>(1-\e)\lambda,\, Sf_1(x)\le \e\lambda  \}.
\end{align*}

Indeed, the inequality $\1_{Q}(x)Sf(x) \geq Sf_{1}(x)$ is trivial. Next,  if  $f^{*}(x)>2\lambda$ for $x \in Q$ then it  follows that there exists $\tilde{Q} \subset Q$, $\tilde{Q} \in \mathcal{D}$ with $x \in \tilde{Q}$ such that $|\langle f\rangle_{\tilde{Q}}|>2\lambda$, and therefore 
$$
f^{*}_{1}(x) \geq |\langle f\rangle_{\tilde{Q}} - \langle f \rangle_{Q}| \geq |\langle f\rangle_{\tilde{Q}}| - |\langle f \rangle_{Q}|>  (1-\varepsilon)\lambda. 
$$

Let $\cR$ be the collection of maximal cubes $R\subset Q$ such that $Sf_1(x)\ge \e\lambda$ everywhere on $R$
%, or equivalently 
%%%
%\begin{align*}
%\sum_{\substack{ K\in\cD\\ R\subsetneqq K\subset Q}} | \Delta\ci K f(x) |^2 > (\e\lambda)^2 \qquad x\in R.
%\end{align*}
%%%
 The set $E\ci Q$ does not intersect cubes $R\in\cR$, so if we define $f_2$ as 
\begin{align}
\label{f_2}
f_2:= \sum_{K\in \cD(Q)\setminus \cup_{R\in\cR}\cD(R)} \Delta\ci K f, 
\end{align}
where $\cD(Q):= \{K\in \cD: K\subset Q \}$, then 
\begin{align*}
E^1\ci Q = E^2\ci Q := \{x\in Q:  f_2^*(x)>(1-\e)\lambda,\, Sf_2(x)\le \e\lambda  \}. 
 \end{align*}

Up to this moment the proof was exactly as the proof of \cite[Corollary 3.1]{ChWiWo1985}. If one uses the square function $S_\infty$ (or $S_p$) instead of $S$, on can conclude, that because the term $\|\Delta\ci K f\|_\infty^2 \1\ci K$ is constant on $K$, then by the construction $S_\infty f_2(x) \le \e\lambda$ on $Q$, and then use the corresponding superexponential estimates \cite[Theorem 3.1]{ChWiWo1985} (Theorem \ref{t:genCWWmax} in this paper). 

But for our square function we have no control on how big $Sf_2$ is on cubes $R\in\cR$! Therefore, to apply Theorem \ref{t:genCWWmax} we need first to modify the function $f_2$.

Let $R\in \cR$ and let $\widehat{R} $ be its parent. If we are lucky, and $\widehat{R}$ has only one child $R$ that belongs to $\cR$, we can estimate $|\Delta\ci{\widehat{R}}f|$ on $R$. Namely, for $x\in\widehat{R}\setminus R$
\begin{align*}
|\Delta\ci{\widehat{R}} f(x) | \le Sf_2 (x) \le \e\lambda, 
\end{align*}
and since $ \Delta\ci{\widehat{R}} f$ has zero average we can estimate that for $x\in R$
\begin{align*}
\left| \Delta_{\widehat{R}} f\right| \le \frac{1-\alpha}{\alpha} \e\lambda .
\end{align*}
But if we are not so lucky, and $\widehat R$ has more than one child that belongs to $\cR$, then we cannot estimate the value of $\Delta\ci{\widehat{R}} f$ on such children. If $R_k$ are all the children of $\widehat{R}$ that belong to $\cR$, we can definitely estimate the average, 
\begin{align*}
\left|\langle \Delta_{\widehat{R}} f \rangle\ci{\cup_k R_k} \right| \le \frac{1-\alpha}{\alpha} \e\lambda ,
\end{align*}
but because of possible cancellations, the values on $R_k$ can be arbitrarily large. 

So let us change the martingale difference $\Delta\ci{\widehat{R}} f = \Delta\ci{\widehat{R}} f_1 = \Delta\ci{\widehat{R}} f_2$ by replacing its values on the cubes $R_k$ by the average $\langle \Delta_{\widehat{R}} f \rangle\ci{\cup_k R_k}$. The resulting function will still have zero average, so it is also a martingale difference, let us call it $ f_{\widehat{R}}$. 

So, if we define the function $f_3$ by replacing in \eqref{f_2}  the martingale differences $\Delta_{\widehat{R}} f$ by $f_{\widehat{R}}$ for the parents $\widehat{R}$ of cubes $R\in\cR$, then outside of cubes $R\in\cR$ we have $f_2 = f_3$ and $Sf_2=Sf_3$; note also that $f_{\widehat{R}} = \Delta_{\widehat{R}} f_3$. Therefore, 
\begin{align*}
E\ci Q^2\subset E\ci Q^3 := \{x\in Q:  f_3^*(x)>(1-\e)\lambda,\, Sf_3(x)\le \e\lambda  \}. 
\end{align*}
But $Sf_3$ can be estimated, 
\begin{align*}
\|Sf_3(x)\|_\infty^2\le (\e\lambda)^2 + (\e\lambda)^2 \frac{(1-\alpha)^2}{\alpha^2} = (\e\lambda)^2 \frac{1-2\alpha + 2\alpha^2}{\alpha^2} , 
\end{align*}
and we can apply the superexponential distributional estimates from Theorem \ref{t:genCWWmax} to get that 
\begin{align*}
|E\ci Q|\le |E\ci Q^3| \le 2 \exp\left( - \frac{\alpha^3 (1-\e)^2\lambda^2}{ (1-2\alpha+2\alpha^2)\e^2\lambda^2} 
\right) |Q|
\end{align*}
%%
%Summing over all $Q\in\cQ$ we get the conclusion of the lemma.  
\end{proof}

Slightly abusing notation for $E\subset \cX$ and $w\in L^1$ we will denote 
\begin{align*}
w(E):= \int_E w(x) \dd x
\end{align*}
\begin{lemma}
	\label{l:w(E)}
Suppose that for an $\alpha$-homogeneous filtration and $w\in L^1$, $w\ge 0$ we have for a cube $Q_0\in\cD$
\begin{align*}
\langle w^* \rangle\ci{Q_0} \le A \langle w \rangle\ci{Q_0}  .
\end{align*}
Let $E\subset Q_0$ be a union of some cubes in $\cD(Q_0)$. Then 
\begin{align}
\label{w(E)}
 w(E) \le 2A \frac{\ln (2/\alpha)}{\ln(|Q_0|/|E|)} w(Q_0)
\end{align}
\end{lemma}
\begin{proof}
Define $E_0:=E$ and
\begin{align*}
E_k:= \{x\in \cX : (\1\ci E)^*(x)  > (\alpha/2)^k\}, \qquad k=1, 2, \ldots , \lfloor \log_{2/\alpha}(|Q_0|/|E|)   \rfloor :=N_\alpha . 
\end{align*}

Let $\cE_k$ be the collection of maximal cubes $Q\subset E_k$. It is easy to see that for any $Q\in\cE_k$, $k=1, 2, \ldots, N_\alpha$
\begin{align}
\label{sparse}
\sum_{R\in \cE_{k-1}, R\subset Q} |R| \le \frac12 |Q|. 
\end{align}

Let $f:= w \1\ci E$. Since for every cube $Q\in\cE_k$
\begin{align*}
|f^*(x)| \ge |Q|^{-1}\sum_{K\in\cE_0, \, K\subset Q} w(K), \qquad \forall x\in Q, 
\end{align*}
we conclude summing over all $Q\in\cE_k$ and using \eqref{sparse} that 
\begin{align*}
\int_{E_k\setminus E_{k-1}} f^*(x)\dd x \ge \frac12 w(E), 
\end{align*} 
and clearly $\int_{E} f^*(x)\dd x\ge w(E)$. Therefore
\begin{align*}
(1+N_\alpha/2)w(E)\le \int_{Q_0} f^* \dd x \le \int_{Q_0} w^* \dd x \le A \int_{Q_0} w \dd x =A w(Q_0).  
\end{align*}
Noticing that 
\begin{align*}
(1+N_\alpha/2)\ge (1+N_\alpha)/2 \ge \frac12 \log_{2/\alpha}(|Q_0|/|E|) = \frac{\ln(|Q_0|/|E|)}{2 \ln (2/\alpha)}
\end{align*}
concludes the proof. 
\end{proof}

\subsection{Proof of Theorem \ref{t:weighted}}
%\begin{proof}[Proof of Theorem \ref{t:weighted}]
The standard good $\lambda$ inequalities reasoning implies that if 	
\begin{align}
\label{w-good_lambda}
w(\{ x\in \cX:  f^*(x)  >2\lambda,  Sf(x)\le \e\lambda  \}) \le 2^{-(p+1)} w(\{x\in\cX: f^*>\lambda \}), 
\end{align}
then 
\begin{align*}
\|f^*\|\ci{L^p(w)} \le  \e^{-1} 2^{(p+1)/p}  \| Sf \|\ci{L^p(w)}
\end{align*}

Recall that $\cQ$ is the collection of maximal cubes $Q\in\cD$ such that $|\langle f \rangle\ci{Q}|>\lambda$. For $Q\in\cQ$ denote
\begin{align*}
E\ci Q:= \{ x\in Q:  f^*(x)  >2\lambda,  Sf(x)\le \e\lambda  \}. 
\end{align*}
To prove \eqref{w-good_lambda} it is sufficient to show that 
 for all $Q\in\cQ$ 
\begin{align}
\label{w-good_lambda_Q}
w(\{ x\in Q:  f^*(x)  >2\lambda,  Sf(x)\le \e\lambda  \}) \le 2^{-(p+1)} w(Q) ;
\end{align}
taking the sum over all $Q\in\cQ$ give us \eqref{w-good_lambda}. 

By Lemma \ref{l:w(E)} the above inequality \eqref{w-good_lambda_Q}  holds if 
\begin{align}
\label{eps-A-alpha-01}
2^{-(p+1)} \ge 2A \frac{\ln (2/\alpha)}{\ln(|Q|/|E\ci Q|)}\,,
\end{align}
and estimating $|Q|/|E\ci Q|$ using  Lemma \ref{l:Good_la-01} we see that \eqref{eps-A-alpha-01} holds if 
\begin{align}
\label{eps-A-alpha}
A 2^{p+2}\ln(2/\alpha) \le \frac{\alpha^3(1-\e)^2}{(1-2\alpha + 2\alpha^2)\e^2} -\ln 2 . 
\end{align}
If we put
\begin{align*}
\e^{-1}=  A^{1/2} 2^{(p+3)/2}\alpha^{-3/2} (\ln(2/\alpha))^{1/2}, 
\end{align*}
the inequality \eqref{eps-A-alpha} is satisfied, and the theorem is proved. 
% with 
%\[
%C(\alpha, p )= 2^{1/p}2^{(p+3)/2}\alpha^{-3/2} \ln^{1/2}(2/\alpha).
%\] 
%\end{proof}
\hfill \qed

\subsection{Unweighted \texorpdfstring{$L^p$}{L^p} estimates for the square function}
The standard technique of good $\lambda$ inequalities also allows one to get the unweighted $L^p$ estimates for the square function. Note that the estimate behaves as $Cp^{1/2}$ as $p\to \infty$

\begin{theorem}
	\label{t:L^p-est}
	For the classical square function $S$ and any $p\in(0,\infty)$ 
	\begin{align}
	\label{L^p-est-01}
	\|f^*\|\ci{L^p} \le 4\cdot 2^{1/p} \alpha^{-3/2} (p+2)^{1/2}\|Sf \|\ci{L^p} .
	\end{align}
\end{theorem}

\begin{proof}
	The standard good $\lambda$ inequalities reasoning implies that if 	
	\begin{align}
	\label{good_lambda-02}
	\left|\{ x\in \cX:  f^*(x)  >2\lambda,  Sf(x)\le \e\lambda  \}\right| \le 2^{-(p+1)} \left|\{x\in\cX: f^*>\lambda \}\right|, 
	\end{align}
	then 
	\begin{align}
	\label{good-lambda_conclusion}
	\|f^*\|\ci{L^p} \le  \e^{-1} 2^{(p+1)/p}  \| Sf \|\ci{L^p}. 
	\end{align}
	Lemma \ref{l:Good_la-01} shows that  \eqref{good_lambda-02}  holds if $\e$ satisfies 
	\begin{align*}
	\frac{\alpha^3 (1-\e)^2}{(1-2\alpha + 2\alpha^2)\e^2} \ge p+2. 
	\end{align*} 
	Since  $\e^{-1}= 2\alpha^{-3/2} (p+2)^{1/2}$ satisfies the above inequality, \eqref{good-lambda_conclusion} immediately gives \eqref{L^p-est-01}.
\end{proof}

\subsection{Comparison with the estimates from \cite{DIPTV_2017}}

It was proved in \cite[Theorem 3.4]{DIPTV_2017} that for the $n$-adic filtration, i.e.~for a filtration where each atom $Q$ has exactly $n$ children of equal measure, the following estimate holds for all $f$ with $\langle f\rangle\ci \cX = 0$:
\begin{align*}
\| f \|\ci{L^2(w)} \lesssim n [w]_{A_\infty\ut{scl}}^{1/2}\|S f\|\ci{L^2(w)} . 
\end{align*}
Here $\cX=[0,1)$ (as we discussed it above in the introduction, we can always assume that without loss of generality), and  $A_\infty\ut{scl}$ means the \emph{semiclassical} $A_\infty$-condition, 
\begin{align}
\label{A_infty-scl}
\sup_{Q\in\cD} {\langle M\ci Q w \rangle\ci Q}/{\langle w \rangle\ci Q} =: [w]_{A_\infty\ut{scl}} <\infty, 
\end{align}
where $M\ci Q$ is the localized classical Hardy--Littlewood maximal function, one the real line 
\begin{align}
\label{H-L_max}
M\ci Q f(x) =\sup \{\langle |f| \rangle\ci I : I\subset Q, I\ni x \}, 
\end{align}
and the supremum is taken over all intervals $I\subset Q$. We assumed here, that our probability space is the unit interval $[0,1)$, which as it was discussed above we can do without loss of generality. 

It can be seen from the proof of \cite[Theorem 3.4]{DIPTV_2017} that for an $\alpha$-homogeneous filtration 
\begin{align}
\label{DIPTV-lb}
\| f\|\ci{L^2(w)} \lesssim \alpha^{-1} [w]_{A_\infty\ut{scl}}^{1/2}\|S f\|\ci{L^2(w)},   
\end{align}
for any $f$ with $\langle f \rangle\ci \cX=0$. 

It was shown in \cite[Theorem 3.3]{DIPTV_2017} that for a general non-homogeneous filtration the  $A_\infty$ condition on the weight $w$ is not sufficient for the estimate $\|f\|\ci{L^2(w)} \le C  \|Sf\|\ci{L^2(w)}$, even if one considers the (stronger) classical $A_\infty$ condition on the real line; in the classical $A_\infty$ condition on the line the supremum in \eqref{A_infty-scl} is taken over all intervals $Q$, not only $Q\in\cD$. Therefore, a dependence on $\alpha$ that blows out when $\alpha\to 0$ must be present in the estimates \eqref{DIPTV-lb}, \eqref{weighted_f-S}.

It can be shown that 
\begin{align}
\label{A_infty-mart-scl}
[w]_{A_\infty\ut{scl}}  \lesssim \alpha^{-1}[w]_{A_\infty}; 
\end{align}
the proof is essentially given in the proof of \cite[Chapter 3, Theorem 3.8]{Bennett-Sharpley_1988}; an obvious  adaptation gives a proof of \eqref{A_infty-mart-scl}.  

Therefore, the estimate \eqref{H-L_max} implies the following estimate in terms of the martingale $A_\infty$ characteristic $[w]\ci{A_\infty}$:
\begin{align}
\label{DIPTV-lb-01}
\| f\|\ci{L^2(w)} \lesssim \alpha^{-3/2} [w]_{A_\infty}^{1/2}\|Sf\|\ci{L^2(w)}.  
\end{align}
Note that this estimate has the same exponent at $\alpha$ as the estimate \eqref{weighted_f-S}; the estimate \eqref{weighted_f-S} has an extra factor $(\ln(2/\alpha))^{1/2}$, but it deals with a bigger function $f^*$ instead of $f$. So, we can say that the asymptotic in $\alpha$ in \eqref{DIPTV-lb} and \eqref{DIPTV-lb-01} are essentially equivalent.

But we do not know how to get using the superexponential estimate \eqref{t:genCWW} a better exponent at $\alpha$ in \eqref{weighted_f-S} by replacing $[w]\ci{A_\infty}$ by the bigger $A_\infty$ characteristic  $[w]_{A_\infty\ut{scl}}$.

\section{From weighted estimates to the superexponential distribution inequalities. } 
\label{s:weight to superexp}

A standard and well-known to experts reasoning allows one to obtain from the proved in \cite{DIPTV_2017} estimate $\|f\|\ci{L^2(w)} \le C [w]_{A_\infty\ut{scl}}^{1/2} \|f\|\ci{L^2(w)}$, see \eqref{DIPTV-lb} above,  the superexponential distribution inequalities  \eqref{t:genCWWmax} (with appropriate parameters). 

For the reader's convenience we summarize the known results here. 
\subsection{Preliminary lemmas}
Recall that a weight $w$ on the real line (or its subinterval) satisfies the (classical) $A_1$ condition if 
\begin{align}
\label{A_1}
\sup_{x\in\R}  \frac{M w (x) }{  w (x)}  := [w]\ci{A_1}<\infty;
\end{align}
here $M$ is the classical Hardy--Littlewood maximal function \eqref{H-L_max}. 
%, and the \emph{supremum} is taken over all intervals $I$. 

It is well-known and trivial to see that the $A_1$ condition is stronger than $A_\infty$ condition considered in this paper and that the $A_1$ characteristic $[w]\ci{A_1}$ majorates all the $A_\infty$ characteristics. 

Recall that the norm of the Hardy--Littlewood maximal operator $M$ in $L^p(\R)$, $1<p\le\infty$, is estimated by $C\ci M p'$, where $C\ci M$ is an absolute constant and $p'$ is the dual H\"{o}lder exponent, $1/p +1/p'=1$. 

\begin{lemma}
\label{l:RubioD-F}
Let measurable functions $f$ and $g$ on $\R$ (or on an interval $I\subset \R$) be such that for all $A_1$ weights 
\begin{align}
\label{A_1-est-01}
\|f\|\ci{L^2(w)} \le B [w]\ci{A_1}^{\kappa} \|g\|\ci{L^2(w)} \,, 
\end{align}
for some $B<\infty$, $\kappa>0$. 

Then for all $p\ge 2$ we have in the non-weighted $L^p$
\begin{align}
\label{L^p-est}
\|f\|\ci{L^p} \le \sqrt{2} \,C\ci M B p^\kappa \|g \|\ci{L^p}
\end{align}
\end{lemma}

\begin{proof}
The proof is the standard Rubio de Francia extrapolation argument, as it is presented on pp.~356--357 of \cite{Feff-Pipher_1997-multiparameter}. %Should we present it here????
For $p\ge 2$ and $r=p/2$
\begin{align*}
\|f\|\ci{L^p}^2 =\|f^2\|\ci{L^r} = \sup \left\{ \int_\cX |f|^2  \f : \f\ge 0,\, \|\f\|\ci{L^{r'}} =1   \right\}. 
\end{align*}
Take $\f \in L^{r'}$, $\|\f\|\ci{L^{r'}} =1$. Let $M^1\f=M\f$, $M^{2} \f= M(M\f)$, $M^{n+1} \f = M(M^n \f)$. Define the weight $w$ by 
\begin{align*}
w:= \f + \sum_{n=1}^{\infty} (C\ci M p)^{-n} M^n\f 
\end{align*} 
Since $\|M \f\|\ci{L^{r'}}\le C\ci M r = C\ci M p/2$, we can conclude that 
\begin{align*}
&(C_{M} p)^{-1}\|M\f\|\ci{L^{r'}} \le 2^{-1}\|\f\|\ci{L^{r'}}, \\
&(C_{M} p)^{-n-1}\|M^{n+1}\f\|\ci{L^{r'}} \le 2^{-1}(C_{M}p)^{-n}\| M^{n} \f\|\ci{L^{r'}} = 2^{-(n+1)}  \|\f\|_{L^{r'}}, \quad n\ge 1, 
\end{align*}
so 
\begin{align}
\label{norm_w_in_L^r'}
\|w\|\ci{L^{r'}} \le \|\f\|\ci{L^{r'}} \sum_{k=0}^\infty 2^{-n} = 2 \|\f\|\ci{L^{r'}}, 
\end{align}
and therefore $w$ is finite a.e.

It is easy to see that $Mw \le C\ci M p w$, so $w$ is an $A_1$ weight with $[w]\ci{A_1} \le C\ci M p$ (we need the a.e.~boundedness of $w$ to avoid the case $w\equiv +\infty$). Continuing with the argument, we get 
\begin{align*}
\int_\cX |f|^2  \f \dd x \le \int_\cX |f|^2  w \dd x \le B^2 [w]\ci{A_1}^{2\kappa} \|g\|\ci{L^2(w)}^2 \le B^2 (C\ci M p )^{2\kappa} \|g\|\ci{L^2(w)}^2 . 
\end{align*}
Estimating 
\begin{align*}
\|g\|\ci{L^2(w)}^2 = \int_\cX |g|^2 w \le \| g^2\|\ci{L^r} \|w\|\ci{L^{r'}} = \|g\|\ci{L^p}^2 \|w\|\ci{L^{r'}}
\le 2 \|g\|\ci{L^p}^2,
\end{align*}
(the last inequality holds because of \eqref{norm_w_in_L^r'}) we get that 
\begin{align*}
\int_\cX |f|^2  \f \dd x \le  2 B^2 (C\ci M p )^{2\kappa} \|g\|\ci{L^p}^2, 
\end{align*}
and taking supremum over all $\f\ge 0$, $\|\f\|_{L^{r'}}=1$ we get the conclusion of the lemma. 
\end{proof}

\begin{lemma}
\label{l:L^p-est->superExp}
Let $\kappa>0 $ and let $f\in L^1(\cX)$, $g\in L^\infty(\cX)$, $|\cX|=1$ be such that for all  $p\ge 1/\kappa$
\begin{align}
\label{L^p-L^infty}
\|f\|\ci{L^p} \le B p^\kappa \|g\|\ci{L^\infty}. 
\end{align}
Then for $\gamma= (2e)^{-1}\kappa B^{-1/\kappa}\|g\|\ci{L^\infty}^{-1/\kappa}$
\begin{align}
\label{superexp_integrability}
\int e^{\gamma |f|^{1/\kappa}} \le C<\infty
\end{align}
where
\begin{align}
\label{C_in_superexp_integrability}
C= 1 + \sum_{n=1}^\infty \frac{(n/(2e))^n}{n!} <\infty
\end{align}
\end{lemma}

%The estimate \eqref{L^p-est},implies that 
%%%
%\begin{align}
%\label{supexp-02}
%\left| \{ x\in \cX: f^*(x)  > \lambda    \}  \right| \leq 2 e^{-\alpha\lambda ^{2}/ \|S f\|_{\infty}^2}.
%\end{align}

\begin{remark*}
	The series in \eqref{C_in_superexp_integrability} trivially converges by the ratio test. 
\end{remark*}
\begin{proof}[Proof of Lemma \ref{l:L^p-est->superExp}]
	Representing $e^{\gamma |f|^{1/\kappa}} $ as a power series and using \eqref{L^p-L^infty} we get 
	\begin{align*}
	\int e^{ \gamma |f|^{1/\kappa}} \dd x = 1 + \sum_{n=1}^\infty \frac{\gamma^n}{n!} \| f\|_{L^{n/\kappa}}^{n/\kappa}
	& \le 1 + \sum_{n=1}^\infty \frac{\gamma^n}{n!} B^{n/\kappa} \left(\frac{n}{\kappa}\right)^{n} \| g\|_{L^\infty}^{n/\kappa}
	\\
	& = 1 + \sum_{n=1}^\infty \frac{(n/(2e))^n}{n!}  =C<\infty;
	\end{align*}
	as it was mentioned in the remark above, the convergence of the series follows easily from the ratio test. 
\end{proof}

\begin{remark}
	\label{r:Lp-est->superexp}
	As one can see from the proof of Lemma \ref{l:L^p-est->superExp} it is sufficient to assume that \eqref{L^p-L^infty} holds only for all  sufficiently large $p$. Indeed, suppose the estimate 
	\begin{align*}
	\| f \|_{L^{n/\kappa}}^{n/\kappa} \le B^{n/\kappa} \| g\|_{L^\infty}^{n/\kappa}
	\end{align*}
	holds only for $n\ge N$. But by monotonicity  for $1\leq \ell \leq N$ we have 
	\begin{align*}
	\| f \|_{L^{\ell/\kappa}}^{\ell/\kappa}\leq 1+\| f \|_{L^{N/\kappa}}^{N/\kappa}  \le 1+ B^{N/\kappa} \| g\|_{L^\infty}^{N/\kappa},
	\end{align*}
	so one can estimate
	\begin{align*}
	\int e^{ \gamma |f|^{1/\kappa}} \dd x \le e^{\gamma} + N \frac{(N/(2e))^N}{N!} +   \sum_{n=N+1}^\infty \frac{(n/(2e))^n}{n!} =: C_1 <\infty. 
	\end{align*}
\end{remark}

\subsection{Obtaining superexponential estimates from \eqref{DIPTV-lb}}
Taking $\kappa=1/2$ and applying lemma \ref{l:RubioD-F}  to the estimate \eqref{DIPTV-lb}  we get that 
for $f\in L^1(\cX)$ with $\langle f \rangle\ci\cX=0$ and $p\ge 2$
\begin{align}
\label{f^*le_Sf-old}
\|f^*\|\ci{L^p} \lesssim \alpha^{-1} p^{1/2} \|Sf\|\ci{L^p}.
\end{align}
Lemma \ref{l:L^p-est->superExp} then gives us
\begin{align}
\label{superexp_integrability-01}
\int_\cX e^{c\alpha^2|f|^2/\|Sf\|_\infty^2} \dd x \le C<\infty. 
\end{align}
This inequality in turn implies the superexponential distribution inequality
\begin{align}\label{superexp_distr-01}
\left| \{ x\in \cX:| f(x) | > \lambda    \}  \right| \leq  C e^{- c\alpha^2\lambda ^{2}/ \|S f\|_{\infty}^2}.
\end{align}

\begin{remark}
	\label{r:dependence_on_alpha}
Estimate \eqref{f^*le_Sf-old} gives us better dependence on the homogeneity parameter $\alpha$ then the estimate \eqref{L^p-est-01} that we got from the superexponential estimate \eqref{distr_of_psi^*}; the drawback of \eqref{f^*le_Sf-old} is that it holds only for $p\ge2$, versus $0<p<\infty $ in \eqref{L^p-est-01}. 

However, the superexponential estimate \eqref{superexp_distr-01} we got from there gives us worse dependence on $\alpha$ then the estimates \eqref{ext},  \eqref{distr_of_psi^*}.
\end{remark}

\subsection{A discussion}

It is a remarkable result of Wang~\cite{GWang} that for any  $3 \leq p <\infty$   and any {\em conditionally symmetric martingale} $f$ we have the sharp bound
\begin{align}\label{bilyk}
\| f\|_{p} \leq  R_{p} \| S f\|_{p}, \quad \mathbb{E}f=0, 
\end{align}
where $R_{p}$ is the  rightmost zero of the Hermite function $H_{p}(x)$. Here $H_{p}(x)$ is the solution of the Hermite differential equation 
\begin{align*}
H''_{p}-xH'_{p}+pH_{p}=0
\end{align*}
which grows relatively slowly at infinity, i.e., $H_{p}(x)= x^{p}+o(x^{p})$ as $x \to \infty$. It is known that $R_{p}=O(\sqrt{p})$ as $p \to \infty$. In our context  ``conditionally symmetric'' martingale  means % a function $f \in L^{1}$ such  
that  for any dyadic cube $Q$ the martingale differences $\Delta\ci Qf$ and $-\Delta\ci Q f$ have the same distribution function. 
%can be partitioned into disjoint union  
%$$
%Q= Q_{1}\cup \ldots \cup Q_{m} \cup Q'_{1}\cup\ldots \cup Q'_{m},  \quad Q_{j}, Q'_{j} \in \mathrm{ch}\, Q,\; |Q_{j}|=|Q'_{j}|, \quad  j=1,\ldots,m, 
%$$
%so that the values of  $\Delta_{Q}f$ on each  $Q_{j}$ and $Q'_{j}$  have the opposite signs. 
  
 Clearly that for the standard dyadic filtration  on $[0,1)^{d}$ not all martingales are conditionally symmetric. In fact, for such a filtration any $f\in L^1$ defines a conditionally symmetric martingale   if and only if  $d=1$.  

Thus we obtain that if $f$ is conditionally symmetric martingale with  $\mathbb{E}f=0$ and $\|Sf\|_{\infty}<0$,  then estimate (\ref{bilyk}), combined with the facts $R_{p} = O(\sqrt{p})$, $\|Sf\|_{p}\leq \|Sf\|_{\infty}$,  and Lemma~\ref{l:L^p-est->superExp} implies the superexponential bound 
\begin{align}\label{bbcor}
\left| \{ x\in \cX:| f(x) | > \lambda    \}  \right| \leq  C e^{- c \lambda ^{2}/ \|S f\|_{\infty}^2}.
\end{align}

\bigskip

Finally we remark that for the classical dyadic square function on $[0,1)^{d}$ there is no estimate of the type (\ref{bilyk}), i.e.,   
$$
\| f\|_{p} \leq C \sqrt{p} \|S f\|_{p}, \quad \mathbb{E} f=0
$$
with $C$ independent of $d$, and $p>N$ for some positive $N$. Indeed,  otherwise this would give us superexponential bound (\ref{bbcor}) independent of dimension $d$, but we will see in the next section  that the estimate 
$$
|\{x \in [0,1)^{d}\, : f(x) > \lambda \}| \leq e^{-\lambda^{2}/(2^{d}\|Sf\|_{\infty}^{2})}
$$
is sharp when dimension $d$ goes to infinity, see Theorem~\ref{t:sharpness-02}.

%Once again we see the obstacle  in higher dimensions because dyadic martingale constructed by $f$ on $[0,1)^{d}$ for $d\geq2$ is not conditionally symmetric and therefore we do not know if the estimate \eqref{bilyk} still holds. On the other hand the inequalities of the type \eqref{hoeff} do not help us because we can only control  the quantity $\|Sf\|_{\infty}$ which, by our previous example, happens to be completely different from $\|S_{\infty} f\|_{\infty}$. 

\section{Sharpness of Theorems~\ref{trpi0}, \ref{t:genCWW}}
\subsection{Sharpness of theorem \ref{trpi0}}
We remark that the constant $\alpha$  in the left hand side of  \eqref{trpi1} is sharp when $\alpha=1/2$, in this case  even Theorem~\ref{Slavin} is sharp (see the proof of Lemma~3.3 in  ~\cite{Slavin1}).   It is also true asymptotically as $\alpha\to 0$, in the sense that $\alpha$ in \eqref{trpi1} cannot be replaced by $C\alpha$ with $C>1$ independent of $\alpha$ and the estimate being true for all $\alpha>0$. Indeed, this is true even for the square function $S$ in the $n$-adic filtration (here $\alpha=1/n$). To verify the sharpness of the estimate for the  case $n \to \infty$, which corresponds to $\alpha\to 0$,  assume the contrary that we have 
\begin{align}\label{bolo11}
\int_{Q_{0}} \exp\left(f - C\frac{2^{d}S(f)^{2}}{4} \right) \leq \exp\left(\int_{Q_{0}}f\right)
\end{align}
for some $C$ with $0<C<1$ and $d$ sufficiently large. Consider the dyadic cubes $Q\in  \mathcal{D}_{1}$ of the first generation, i.e., $|Q|=2^{-d}$. We remind that $Q_{0}=[0,1)^{d}$.  In total we have $2^{d}$ such cubes. Take any $f$ such that $f=2^{1-d}$ on $2^{d}-1$ such cubes, and $f=2(2^{-d}-1)$ on the remaining last cube. Clearly 
$$
\int_{[0,1]^{d}} f = 2^{1-d}(2^{d}-1)2^{-d}+2(2^{-d}-1)2^{-d}=0.
$$ 
 %Consider the test function $f=d_{1}$ where $d_{1}$ takes only two values $c_{1}=2p$ and $ c_{2} =2(p-1)$ with probabilities $1-p$ and $p$ correspondingly. Clearly $d_{1}$ is $p$-homogeneous martingale difference. 
  Inequality \eqref{bolo11} is equivalent to  showing that 
\begin{align}\label{rm1}
(1-\alpha) \exp\left(2\alpha - C\alpha\right) + \alpha \exp\left(2(\alpha-1) - C\frac{(\alpha-1)^{2}}{\alpha}\right)\leq 1, 
\end{align}
where $\alpha=2^{-d}$. 
By Taylor's formula the left hand side of \eqref{rm1} at point $\alpha=0$ can be written as $1+(1-C)\alpha+O(\alpha^{2})$. Therefore we  see that  the condition $C\geq 1$ is necessary for the inequality to hold. 

\bigskip

Our next remark is that for each fixed $\alpha \in (0,1/2)$ one can replace the constant $\alpha$ in the denominator  in exponent of the left hand side of  \eqref{trpi1} by a better constant $C(\alpha)$ with $C(\alpha)\geq \alpha$. Repeating our arguments without any changes one can show that the largest constant $C=C(\alpha)$ one can have in the inequality 
\begin{align*}
\mathbb{E}   \exp\left(f- \frac{S(f)^{2}}{4C} \right) \leq \exp\left(\mathbb{E} f\right).
\end{align*}
coincides with the largest number $C$ which satisfies the inequality 
\begin{align*}
p_{1} e^{c_{1}- \frac{c_{1}^{2}}{4C}} + p_{2} e^{c_{2}- \frac{c_{2}^{2}}{4C}}\leq 1,
\end{align*}
for all $c_{1}, c_{2} \in \mathbb{R}$, any $p_{1}, p_{2}\geq \alpha$ with $p_{1}+p_{2}=1$ and $p_{1}c_{1}+p_{2}c_{2}=0$. The best possible such constant $C=C(\alpha)$  can be found, but it does not have a nice form, and we did not find it useful for the practical purposes. However, we just showed that $C(\alpha)\geq \alpha$ and the latter estimate interpolates in a sharp way the endpoint cases  for the range $\alpha \in (0,1/2]$.

%Of course Theorem~\ref{kk1} is the restatement of the claim that for any finite set $X$ equipped with uniform counting measure, and any $f :X \to \mathbb{R}$ with $\mathbb{E} f=1$ we have $\mathbb{E} e^{-f^{2}/|X|} \leq e^{-1/|X|}$.

\subsection{Sharpness of Theorem \ref{t:genCWW}}

\begin{theorem}
	\label{t:sharpness-02}
Theorem \ref{t:genCWW} is sharp when $\alpha \to 0$, meaning that $\alpha$ in the exponent in the right hand side of \eqref{ext} cannot be replaced by $c\alpha$ with $c>1$. 

More precisely, given $c>1$, for all sufficiently small $\alpha$ and for all $\lambda>\lambda(\alpha)$ there exists an $\alpha$-homogeneous filtration and a function $f$, $\|Sf\|_\infty=1$ such that 
\begin{align}
\label{sharpness-02}
|\{ x\in\cX : f(x)>\lambda  \}| > e^{-c\alpha\lambda^2}
\end{align}
\end{theorem}

As a corollary we can see that the superexponential estimate like \eqref{ext} fails for non-homogeneous filtrations. 

\begin{proof}[Proof of Theorem \ref{t:sharpness-02}.]
%For small $\alpha >0$ let us construct an $\alpha$-homogeneous filtration on $\cX =I_0 = [0,1)$ as follows.  

Let us first describe the construction for $\alpha=1/n$, $n>2$, the construction for general $\alpha$ will be the same with obvious modifications. So, let $\alpha=1/n$. Consider the standard $n$-adic filtration. 

Any set $E\in\cF_k$ an be split into two disjoint sets $E^\alpha, E^\beta \in\cF_{k+1}$ such that
\begin{align*}
|E^\alpha| = \alpha |E|, \qquad |E^\beta| =(1-\alpha)|E|;
\end{align*}
the set $E^\alpha$ can be constructed, for example, by picking one of the intervals $I\in \cD_{k+1}$ out of each interval $J\in\cD_k$ comprising $E$, and the set $E^\beta$ will be the rest. 

We start with $E_0=\cX$, and  split it as above into the sets $E_0^\alpha, E_0^\beta\in\cF_1$, and denote $E_1:=E_0^\beta$. We then repeat the same procedure for $E_1$, and inductively construct sets $E_k\in\cF_k$ by defining
$E_{k+1}:= E_k^\beta$. 

The same construction works for general $\alpha$. We start with $\cX$, which is the only element of $\cD_0$, and construct the sets $\cD_k$ inductively. Namely, to get $\cD_{k+1}$ we split each $I\in\cD_k$ into finitely many intervals $I_j$ such that $|I_1|= \alpha |I|$ and $|I_j|\ge \alpha |I|$ for $j\ge2$. Clearly, that in this situation we can split a set $E\in \cF_k$ into the sets $E^\alpha, E^\beta\in \cF_{k+1}$ as above, and construct inductively the sets $E_k$.  

For now let $\alpha$ be fixed. Define the functions 
\begin{align*}
d_k :=\1\ci{E_k^\beta} - \frac{1-\alpha}{\alpha}\1\ci{E_k^\alpha}. 
\end{align*}  
Clearly the functions $d_k$ are $\cF_{k+1}$-measurable and $\bE_k d_k=0$, so they are martingale differences. 

In what follows we will use the notation $f\sim g$ for the \emph{sharp} asymptotic equivalence of $f$ and $g$, meaning $\lim(f/g) = 1$; when it is not evident from the context, the direction for the limit will always be specified.

Take $N>1/\alpha^2$. Then an easy observation shows that for $g\ci N:= \sum_{k=0}^{N-1} d_k$
\begin{align*}
\|S g\ci N\|_\infty & = \left( N-1 +(1-\alpha)^2/\alpha^2  \right)^{1/2}  \sim N^{1/2} \qquad\text{as } N\to \infty
\intertext{ and that }
g\ci N(x) &\ge N \qquad \text{on a set of measure } (1-\alpha)^N
\end{align*}

For sufficiently large $\lambda>0$ define $N(\lambda)$ to be the smallest integer $N$ such that $N/\|Sg\ci N\|_\infty \ge\lambda$; clearly 
\begin{align*}
N(\lambda) \sim \lambda^2, \quad N(\lambda)/\|Sg\ci{N(\lambda)}\|_\infty \sim \lambda \qquad \text{as }\lambda\to \infty. 
\end{align*}
Define $f=f_\lambda := g\ci{N(\lambda)} /\|S g\ci{N(\lambda)}\|_\infty$. Then $\|Sf_{\lambda}\|_\infty=1$ and $f_\lambda(x) \ge \lambda$ on a set of measure 
\begin{align}
\label{equiv-02}
(1-\alpha)^{N(\lambda)} \sim (1-\alpha)^{\lambda^2} = e^{\lambda^2\ln(1-\alpha)} \qquad\text{as}\quad \lambda\to\infty. 
\end{align}
Noticing that $\ln(1-\alpha)\sim -\alpha$ as $\alpha\to 0$ concludes the proof. 

Namely, for $c>1$ take $c_1$, $c<c_1<1$. Since $\lim_{\alpha\to 0^+} \alpha^{-1}\ln(1-\alpha)=-1$ we conclude that for all sufficiently small $\alpha$, $0<\alpha\le\alpha_0(c_1)$
\begin{align*}
\ln(1-\alpha)\ge -c_1 \alpha. 
\end{align*}
So, using the estimate \eqref{equiv-02} we see that for a fixed sufficiently small $\alpha$ as above, the estimate  \eqref{sharpness-02} holds for all sufficiently large $\lambda$, i.e.~for all $\lambda>\Lambda(c, \alpha)$. 
\end{proof}

\section*{Acknowledgements} We are very grateful to Dmitriy Bilyk, Andrei Lerner, Ramon van  Handel and Pavel Zatitskiy for references and helpful discussions.

\end{document}